\documentclass[english,a4paper,12pt]{amsart}

\usepackage[T1]{fontenc} 
\usepackage[latin1]{inputenc}
\usepackage{amssymb, babel}
\usepackage{amssymb}
\usepackage{amsmath}
\usepackage{bm}

\usepackage{amssymb}
\usepackage{amsmath}

\newtheorem{theorem}{Theorem}[section] 
 
\newtheorem{corollary}[theorem]{Corollary} 
\newtheorem{proposition}[theorem]{Proposition}
 
\theoremstyle{definition}

\newtheorem{definition}[theorem]{Definition} 

\newtheorem{remark}[theorem]{Remark}

\newtheorem{example}[theorem]{Example}

\title[]{A Lie conformal superalgebra and duality of representations for $E(4,4)$}
\DeclareMathOperator{\str}{str}
\DeclareMathOperator{\ad}{ad}
\DeclareMathOperator{\Cur}{Cur}
\DeclareMathOperator{\diver}{div}

\DeclareMathOperator{\Der}{Der}
\DeclareMathOperator{\Ind}{Ind}
\author{Nicoletta Cantarini}\author{Fabrizio Caselli}\author{Victor Kac}

\begin{document}
	\maketitle
	\thispagestyle{empty}
	\newcommand{\inlinewedge}{\C}
	\newcommand{\displaywedge}{\C}
	
	\def\C{{\mathbb F}}
	\def\R{{\mathbb R}}
	\def\Z{{\mathbb Z}}
	\def\lab{\textrm{{\boldmath $\lambda$}}}
	\def\la{\textrm{{\boldmath $\lambda$}}}
	\def\deb{\textrm{{\boldmath $\de$}}}
	\def\mub{\textrm{{\boldmath $\mu$}}}
	\def\nub{\textrm{{\boldmath $\nu$}}}
	\def\yb{\textrm{{\boldmath $y$}}}
	\def\xb{\textrm{{\boldmath $x$}}}
	\def\zb{\textrm{{\boldmath $z$}}}
	\def\wb{\textrm{{\boldmath $w$}}}
	\def\xib{\textrm{{\boldmath $\xi$ \hspace{-1mm}}}}
	\def\de{\partial}
	\def\sl{\mathfrak{sl}}
	\def\g{\mathfrak{g}}

\maketitle 
\begin{abstract} We construct a duality functor in the category of continuous
representations
of the Lie superalgebra $E(4,4)$, the only exceptional simple linearly
compact
Lie superalgebra, for which it wasn't known. This is achieved
by constructing a Lie conformal superalgebra of type $(4,4)$, for which
$E(4,4)$
is the annihilation algebra. Along the way we obtain an explicit
realization of $E(4,4)$
by vector fields on a $(4|4)$-dimensional supermanifold.
\end{abstract}
\section{Introduction}
In our paper \cite{CCK} we constructed a duality functor on the category of continuous modules over some linearly compact Lie superalgebra $L$. The main assumption on $L$ for this construction is the existence of a Lie conformal superalgebra $R$ of type $(r,s)$, whose annihilation algebra is $L$.

The notion of a Lie conformal superalgebra, as treated in \cite{K1}, corresponds to our notion of a Lie conformal superalgebra of type $(1,0)$.
Recall that the latter is an $\C[\de]$-module $R$ with a $\lambda$-bracket
$R\otimes R\rightarrow \C[\lambda]\otimes R$,
\[a_\lambda b=\sum_{j\geq 0}(a_{(j)}b)\lambda^j/j!,\]
satisfying axioms, similar to the Lie algebra axioms. Here $\lambda$  and \ $\de$ are even indeterminates. For the more general Lie conformal superalgebras $R$ of type $(r,s)$ one takes, instead, $r$ even and $s$ odd
indeterminates (see Section 2 for the precise definition of $R$ and a conformal $R$-module).

Most of the work on representation theory of a Lie conformal superalgebra $R$ of type $(1,0)$ was based on the simple observation that conformal $R$-modules
are closely related to continuous modules over the associated to $R$ annihilation algebra ${\mathcal A}(R)$. This name comes from the fact that 
${\mathcal A}(R)$ consists of the operators that annihilate the vacuum vector of the universal enveloping  vertex algebra of $R$ \cite{K1}.

Recall that, if $R$ is a finitely generated $\C[\de]$-module, then
${\mathcal A}(R)$ is a linearly compact Lie superalgebra
\begin{equation}
{\mathcal A}(R)=R[[y]]/(\de+\de_y)R[[y]],
\label{(1)}
\end{equation}
where $y$ is an even indeterminate, with the (well defined) continuous
bracket
\begin{equation}
[ay^m,by^n]=\sum_{j\geq 0}\binom{m}{j}(a_{(j)}b)y^{m+n-j}.
\label{(2)}
\end{equation}
Since $\de$ commutes with $\de_y$, it defines a continuous derivation of ${\mathcal A}(R)$, and we may consider the so-called extended annihilation algebra
 ${\mathcal A}^e(R)=\C[\de]\ltimes {\mathcal A}(R)$. It is straightforward to see that a
conformal $R$-module $M$ is the same as a continuous ${\mathcal A}^e(R)$-module \cite{CK0}.

In order to go from conformal $R$-modules to continuous ${\mathcal A}(R)$-modules, one needs to assume that $\de$ is an inner derivation of 
${\mathcal A}(R)$:
\begin{equation}
\de=\ad a \,\,\,\mbox{for some}\,a\in
{\mathcal A}(R).
\label{(3)}
\end{equation}
A conformal $R$-module (= ${\mathcal A}^e(R)$-module) $M$ is called coherent
if $(\de-a)M=0$. Thus a continuous module over the linearly compact Lie
superalgebra ${\mathcal A}(R)$ is the same as a coherent conformal $R$-module.

A duality functor on the category of conformal $R$-modules for $R$ of type
$(1,0)$ was constructed in \cite{BKLR}. Therefore, under assumption
(\ref{(3)}), this functor can be transferred to the category of
${\mathcal A}(R)$-modules.

In \cite{CCK} we extended this construction of a duality functor to the category $\mathcal{P}$ of continuous modules with discrete topology over a linearly compact Lie superalgebra $L$ under the assumptions that there exists a Lie conformal superalgebra $RL$ of type $(r,s)$, for which
\begin{equation}
{\mathcal A}(RL)\cong L,
\label{(4)}
\end{equation}
satisfying an assumption analogous to (\ref{(3)}).

The category $\mathcal{P}$ of $L$-modules is similar to the BGG category $\mathcal O$, and, as in category $\mathcal O$, the most important objects in 
$\mathcal P$ are parabolic Verma modules $M(F)$. Recall that, given an
open subalgebra $L_0\subset L$, and a finite-dimensional $L_0$-module $F$, one defines
\[M(F)=\Ind_{L_0}^L F.\]
(In \cite{CCK} these modules are called generalized Verma modules.)
The main result of our paper \cite{CCK} is the computation of the dual to
$M(F)$ $L$-module $M(F)^{\vee}$ if $F$ is finite-dimensional and $L$ has a $D$-conformal structure (see Definition \ref{assumption}), and in particular it satisfies
(\ref{(3)}) and (\ref{(4)}). It turned out that, if $L$ has a $D$-conformal structure,
$M(F)^{\vee}$ is not $M(F^*)$, but $M(F^{\vee})$, where $F^{\vee}$
is the dual $L_0$-module $F^*$  shifted by the following character $\chi$ of $L_0$:
\begin{equation}
\chi(a)=\str(\ad a|_{L/L_0}),\hspace{5mm} a\in L_0.
\label{(5)}
\end{equation}
We proved in \cite{CCK} that many simple linearly compact Lie superalgebras $L$ (classified in 
\cite{K}) with $L_0$ the open subalgebra of minimal codimension, have a $D$-conformal structure, including four (out of five) exceptional $L$. 
However, for the remaining exceptional simple linearly compact Lie superalgebra $L=E(4,4)$, one could construct in a rather natural way a Lie conformal superalgebra of type $(4,0)$ satisfying (\ref{(3)}) and (\ref{(4)}), although it does not provide a $D$-conformal structure for $L$.

The main result of the present paper is a construction of a Lie conformal superalgebra $RE(4,4)$ of type $(4,4)$, whose annihilation algebra is isomorphic to $E(4,4)$, and which provides a $D$-conformal structure for $E(4,4)$ (Corollary \ref{properties} and Corollary \ref{iso}). As a result, we obtain a duality functor on the category $\mathcal P$ for $L=E(4,4)$ and $L_0$ open subalgebra of minimal codimension, for which $M(F)^\vee$ is isomorphic to $M(F^*)$ since the shift $\chi$ defined in (\ref{(5)}) is zero. Along the way we provide an explicit embedding of $E(4,4)$ in the Lie superalgebra $W(4,4)$ of all continuous derivations of the superalgebra of formal power series in four even and four odd indeterminates.

The construction of a duality functor is important for the classification of reducible (i.e., degenerate) parabolic Verma modules. In \cite{CC} and \cite{CCK2} we used this functor in the classification of degenerate parabolic Verma modules over $E(5,10)$. We hope that the duality of modules over $E(4,4)$, constructed in the present paper, will lead to the classification of such modules over $E(4,4)$, the only exceptional linearly compact Lie superalgebra, for which this classification is not known.

The base field $\C$ is assumed to be a field of characteristic 0, and tensor products are assumed to be over $\C$.

\section{Lie conformal superalgebras of type $(r,s)$ and the duality functor for their modules}
Here we recall some fundamental constructions of the theory of Lie conformal superalgebras of type $(r,s)$ and their conformal modules.

Let $r$ and $s$ be two nonnegative integers.
We will use several sets of $r+s$ variables such as $\lambda_1,\ldots,\lambda_{r+s}$, $\de_1,\ldots,\de_{r+s}$, $y_1,\ldots,y_{r+s}$. We will always assume that variables with indices $1,\ldots,r$ are even and variables with indices $r+1,\ldots,r+s$ are odd, and accordingly we let $p_i=\bar{0}$ if $i=1,\ldots,r$ and $p_i=\bar{1}$ if $i=r+1,\ldots,r+s$,
so that $\lambda_i\lambda_j=(-1)^{p_ip_j}\lambda_j\lambda_i$ and
similarly for the $\de_i$ and $y_i$. We will also use bold letters such as $\boldsymbol{\la}$ or $\deb$ or $\yb$ to denote the set of corresponding variables.
We let $\C[\boldsymbol{\la}]=\C[\lambda_1,\ldots,\lambda_r]\otimes {\textrm{\raisebox{0.6mm}{\footnotesize $\bigwedge$}}}(\lambda_{r+1},\ldots,\lambda_{r+s})$ and we similarly define $\inlinewedge[\deb]$ or $\inlinewedge[\yb]$. The completion
$\inlinewedge[[\yb]]$ of $\inlinewedge[\yb]$ is the algebra of formal power series in $\yb$.

If $R$ is a $\mathbb Z/2\mathbb{Z}$-graded vector space we give to $\inlinewedge [\lab]\otimes R$ the structure of a $\mathbb Z/2\mathbb{Z}$-graded $\inlinewedge [\lab]$-bimodule by letting $\lambda_i(P(\lab)\otimes a)=\lambda_iP(\lab)\otimes a$ and $(P(\lab)\otimes a)\lambda_i=(-1)^{p_ip(a)} P(\lab)\lambda_i\otimes a$, where $p(a)\in \mathbb Z/2\mathbb{Z}$ denotes the parity of $a$. We will usually drop the tensor product symbol and simply write $P(\lab)a$ instead of $P(\lab)\otimes a$.

\begin{definition}\label{defsupconf}
	A {\it{Lie conformal superalgebra of type}} $(r,s)$ is a $\mathbb Z/2\Z$ graded $\inlinewedge[\deb]$-bimodule $R$ such that $a\de_i=(-1)^{p_ip(a)}\de_ia$ for all $a\in R$ and $i\in \{1,\ldots,r+s\}$, endowed with a $\lab$-bracket, i.e.\ a $\mathbb Z/2\Z$-graded linear map $R\otimes R\rightarrow \inlinewedge[\lab]\otimes R$, denoted by $a\otimes b\mapsto [a_\lab b]$, that satisfies the following properties:
		\begin{align}& [(\de_i a)_\lab b]=-\lambda_i[a_\lab b],\,\, [a_\lab (b\de_i)]=[a_\la b](\de_i+\lambda_i);& \label{defsupconf1}\\
		&[b_\la a]=-(-1)^{p(a)p(b)}[a_{-\la-\deb}b];&\label{defsupconf3}\\
		 &[a_\la [b_\mub c]] = [[a_\la b]_{\la+\mub}c]+(-1)^{p(a)p(b)}[b_\mub[a_\la c]].&\label{defsupconf4}
		 \end{align}
	 	
\end{definition}
We refer to Property \eqref{defsupconf1}   as  the conformal sesquilinearity, to Property \eqref{defsupconf3} as the conformal skew-symmetry and to
Property \eqref{defsupconf4} as the conformal Jacobi identity.

We introduce 
the following notation. 
If $K=(k_1,\ldots,k_t)$ is any sequence with entries in $\{1,\ldots,r+s\}$ we let
%
\[\la_K=\lambda_{k_1}\lambda_{k_2}\cdots \lambda_{k_t},
\]
and we similarly define $\yb_K$, $\xb_K$ and so on. 
If $K=\emptyset$, we let $\la_K=1$.
We also let $p_K=p_{k_1}+\cdots +p_{k_t}$ and so $p(\la_K)=p_K$.


Starting from a Lie conformal superalgebra $R$ of type $(r,s)$ one can construct a new Lie conformal superalgebra $\tilde R$ of the same type, called the {\it affinization} of $R$ and defined as follows. Let 
$\tilde R=R\otimes \inlinewedge[[\yb]]$. We consider $\tilde R$ as a $\inlinewedge[[\yb]]$-bimodule and also as a $\inlinewedge[\deb_\yb]=\inlinewedge[\de_{y_1},\ldots,\de_{y_{r+s}}]$-bimodule letting
\[
 \de_{y_i} (a\yb_K)=(-1)^{p_ip(a)}a(\de_{y_i} \yb_K)=(-1)^{p_i(p(a)+p_K)}a\yb_K \de_{y_i}
\]
with $\la$-bracket given by
\begin{equation}\label{lambdatilde}
[(\yb_K a)_\la (b \yb_N)]=\big(\yb_K[a_{\la+\deb_{\yb}}b]\big)\yb_N.
\end{equation}
The following proposition holds.
\begin{proposition}\cite[Proposition 2.3]{CCK} The $\inlinewedge[\tilde{\deb}]$-module $\tilde R$ with $\tilde \deb= \deb+\deb_\yb $ and $\la$-bracket given by \eqref{lambdatilde} is a Lie conformal superalgebra of the same type as $R$.	
\end{proposition}

\begin{definition}\label{annihilationalgebra}
	Given a Lie conformal superalgebra $R$ of type $(r,s)$, the {\it{annihilation algebra}} associated to $R$ is the  vector super space
	\[
	\mathcal A(R)=\tilde R /\tilde \deb \tilde R,
	\]
	with bracket given by
	\[
	[\yb _K a,b\yb _N]=[(\yb_K a)_\la (b\yb_N)]_{|\la=0}.
	\]
	
\end{definition}

\begin{proposition}\cite[Proposition 2.5]{CCK}
	$\mathcal A(R)$ is a Lie superalgebra.
\end{proposition}

\begin{definition} A Lie conformal superalgebra $R$ of type $(r,s)$ is called $\Z$-graded if $\inlinewedge[\la]\otimes R=\prod_{d\in\Z} (\inlinewedge[\la]\otimes R)_d$, where $(\inlinewedge[\la]\otimes R)_d$ denotes the
	homogeneous component of degree $d$, and for every homogeneous elements $a,b\in \inlinewedge[\la]\otimes R$ one has:
	\begin{itemize}
		\item[i)] $\deg(\lambda_i a)=\deg (a)-2$;
		\item[ii)] $\deg(\partial_i a)=\deg (a)-2$;
		\item[iii)] $\deg[a_\la b]=\deg(a)+\deg(b)$.
	\end{itemize}
\end{definition}

Notice that if $R$ is a $\Z$-graded Lie conformal superalgebra of type $(r,s)$ then its annihilation algebra $\mathcal A(R)$ inherits a
$\Z$-gradation by setting 
\begin{equation}\label{ten}\deg(a\yb_K)=\deg(a)+2\,\ell(K),\end{equation} where $\ell(K)$ is the number of entries in $K$. 

\begin{definition}
	A  {\it{conformal module}} $M$ over a Lie conformal superalgebra $R$ of type $(r,s)$ is a ${\mathbb Z}/2\Z$-graded 
	$\inlinewedge[\deb]$-module with a ${\mathbb Z}/2\Z$-graded linear map
	$$R\otimes M\rightarrow \displaywedge[\la]\otimes M, ~~
	a\otimes v\mapsto a_{\la}v$$
	such that
	\begin{itemize}
		\item[(M1)] $(\de_i a)_{\la}v=[\de_i,a_{\la}]v=-\lambda_ia_{\la}v$;
		\item[(M2)] $[a_{\la},b_{\mub}]v=a_{\la}(b_{\mub}v)-(-1)^{p(a)p(b)}b_{\mub}(a_{\la})v=
		(a_{\la}b)_{\la+\mub}v.$
	\end{itemize} 
\end{definition}

\begin{definition} The conformal dual $M^\vee$ of a conformal $R$-module $M$ is defined as
$$M^\vee=\{f_{\la}: M\rightarrow \displaywedge[\la]~|~ f_{\la}(\de_i m)=(-1)^{p_ip(f)}\lambda_i f_{\la}(m),
~{\mbox{for all}}~ m\in M\mbox{ and }i=1,\ldots,r+s\},$$
with the structure of $\inlinewedge[\deb]$-module given by $(\de_i f)_{\la}(m)=-\lambda_i f_{\la}(m)$, and
with the following $\la$-action of $R$:
$$(a_{\la}f)_{\mub}m=-(-1)^{p(a)p(f)}f_{\mub-\la}(a_{\la}m), ~a\in R, ~m\in M.$$
Here by $p(f)$ we denote the parity of the map $f_\la$.
\end{definition}

\begin{proposition}\cite[Proposition 3.7]{CCK}\label{Dual1} If $M$ is a conformal $R$-module, then $M^\vee$ is a conformal $R$-module.
\end{proposition}

\begin{proposition}\cite[Proposition 3.8]{CCK}\label{Dual2} Let $T: M\rightarrow N$ be a morphism of conformal $R$-modules i.e. a linear map such that:
\begin{enumerate}
\item $T(\de_i m)=(-1)^{p_ip(T)}\de_iT(m)$,
\item $T(a_{\la}m)=(-1)^{p(a)p(T)}a_{\la}T(m)$,
\end{enumerate}
then the map $T^\vee: N^\vee \rightarrow M^\vee$ given by: $(T^\vee(f))_{\la}m=(-1)^{p(T)p(f)}f_{\la}T(m)$
is a morphism of conformal $R$-modules.
\end{proposition}

\begin{theorem} The assignment $M\mapsto M^\vee$, $T\mapsto T^\vee$ for every conformal $R$-module $M$ and every morphism $T: M\rightarrow N$ of conformal $R$-modules
provides a contravariant functor of the category of conformal $R$-modules.
\end{theorem}
\begin{proof} Due to Propositions \ref{Dual1} and \ref{Dual2}, one only needs to verify that
\begin{itemize}
\item[i)] $(id_M)^\vee=id_{M^\vee}$;
\item[ii)] given two morphisms $T: M\rightarrow N$ and $S:N\rightarrow P$ of conformal $R$-modules, one has: $(ST)^\vee=(-1)^{p(S)p(T)}T^\vee S^\vee$.
\end{itemize} 
This easy check is left to the reader.
\end{proof}

As explained in the Introduction,
conformal $R$-modules are used to study representations of the linearly compact Lie superalgebra $\mathcal A(R)$.
Let $D=\langle \partial_{y_1},\dots, \partial_{y_{r+s}}\rangle$, and consider the semi-direct sum of Lie superalgebras 
${\mathcal A}^e(R)=D\ltimes  \mathcal A(R)$.
This is a natural generalization of the so-called extended annihilation algebra introduced in \cite{CK0}. A key observation made in \cite{CK0} is that conformal $R$-modules are exactly the same as continuous (called conformal in \cite{BKLR}) modules over the extended annihilation algebra. The following proposition extends this observation to our context.

\begin{proposition}\cite[Proposition 3.4]{CCK}
\label{modcorr}
A conformal $R$-module is precisely a continuous module over the Lie superalgebra ${\mathcal A}^e(R)$, i.e.\ a module $M$ such that for every $v\in M$ and every $a\in R$, $(\yb_Ka).v\neq 0$ only for a finite number of $K$. 
The equivalence between the two structures is provided by the following relations:
\begin{itemize}
\item $a_\la v=\sum_{K}(-1)^{p_K}\frac{\la_{\bar K}}{f(K)} (\yb_Ka).v$;
\item $\de_i v=-\de_{y_i}.v$.
\end{itemize}
Here the summation is taken over $K$, viewed up to permutation of its entries, $\bar{K}$ is obtained from $K$ by reversing the order of entries, and $f(K)=\prod_i m_i(K)!$, where $m_i(K)$ is the multiplicity of $i$ in $K$.
\end{proposition}

\begin{definition}\label{regular}
A Lie conformal superalgebra $R$ of type $(r,s)$ is called  \emph{regular} if the subspace $D$ of ${\mathcal A}^e(R)$ consists of inner derivations of ${\mathcal A}(R)$. In this case for each $x\in D$ there exists $a_x\in {\mathcal A}(R)$ such that
$x-a_x$ is a central element of ${\mathcal A}^e(R)$. 
A conformal $R$-module $M$ is called \emph{coherent} if all elements $x-a_x$ act trivially on $M$.
\end{definition}

It follows from Proposition \ref{modcorr} that a coherent $R$-module is precisely a continuous module over $\mathcal A(R)$ (see also \cite[\S 3]{CCK}).
\begin{proposition}\label{coherent}\cite[Proposition 3.7]{CCK} 
Let $R$ be regular and let $M$ be a coherent conformal $R$-module. Then $M^\vee$ is also coherent. 
\end{proposition}

\begin{theorem}
Let $L$ be a linearly compact Lie superalgebra and
let $R$ be a regular Lie conformal superalgebra of type $(r,s)$, such that $L$ is isomorphic to $\mathcal A(R)$. Then the duality functor exists on the category of continuous $L$-modules.
\end{theorem}
\begin{proof} Note that $\mathcal A(R)$ is isomorphic to the factor algebra of ${\mathcal A}^e(R)$ by the ideal generated by all elements $x-a_x$, $x\in D$. Hence the theorem follows from Propositions \ref{modcorr} and \ref{coherent}.
\end{proof}

\begin{remark} Let $L$ be a linearly compact Lie superalgebra, for which $\Der L\subseteq \ad L+\C E$ where $E$ is a diagonalizable operator.
Most of simple $L$, including the five exceptional ones and $W(r,s)$, have this property \cite{CK}. Let $R$ be a Lie conformal superalgebra of type $(r,s)$ such that ${\mathcal A}(R)\cong L$. 
Then $R$ is regular.
\end{remark}

\begin{example}
Let $\g$ be a Lie superalgebra and let $\Cur\,\g=\C[\de]\otimes \g$
be the current Lie conformal superalgebra of type $(1,0)$, for which
$[a_{\lambda}b]=[a,b]$ for $a,b\in 1\otimes\g$.
Then 
\[\mathcal{A}^e(\Cur\, \g)=\C\de_y\ltimes (\g\otimes \C[[y]]).\]
Hence $\de_y$ is not an inner derivation of $\mathcal{A}(\Cur\, \g)$,
and $\Cur\, \g$ is not regular.
\end{example}

In what follows we assume the following technical conditions on a Lie conformal superalgebra $R$ of type $(r,s)$, which turn out to be satisfied in many interesting cases (see \cite{CCK}).

\begin{definition}\label{assumption} Let $R$ be a Lie conformal superalgebra. We say that a linearly compact Lie superalgebra $L$ has a $D$-conformal structure over $R$ if the following conditions are satisfied:
	\begin{enumerate}
	\item $L\cong \mathcal A(R)$;
		\item $R$ is $\Z$-graded;
		\item the induced $\Z$-gradation on $\mathcal A(R)$ has depth at most 3;
		\item the homogeneous components  $\mathcal A(R)_{-1}$ and  $\mathcal A(R)_{-3}$ are purely odd (in particular they can vanish);
		\item the map $\ad:\mathcal A(R)_{-2}\rightarrow \Der (\mathcal A(R))$ is injective and its image is  $D$. 
	\end{enumerate}
\end{definition}

Note that such $R$ is regular.
By condition (5) of Definition \ref{assumption}, if $x\in \mathcal A(R)_{-2}$ then
$\ad(x)$ can be identified with an element in $D$ and hence in  ${\mathcal A}^e(R)$.


\begin{definition}\label{shifteddual} If $\mathfrak{g}$ is a  Lie superalgebra, $\varphi: \mathfrak g\rightarrow \mathfrak{gl}(V)$ is a representation of $\mathfrak{g}$ and  $x\mapsto \chi_x\in \C$ is a character of $\mathfrak{g}$, we let $\varphi^\chi:\mathfrak{g}\rightarrow \mathfrak{gl}(V)$ be given by
	\[\varphi^\chi(x)(v)=\varphi(x)(v)+\chi_x v.
	\]
	\end{definition}
	It is clear that $\varphi^\chi$ is still a representation and we call it the $\chi$-shift of $\varphi$.
	
\begin{definition} \label{rho}
Let $\g$ be a $\Z$-graded Lie superalgebra.
	For $x\in \mathfrak g_0$ let \[\chi_x=\str (\ad(x)_{|\mathfrak g_{<0}}),\]
 where $\str$ denotes supertrace.
\end{definition}

If $V$ is any $\mathfrak g$-module we call the $\chi$-shifted dual of $V$ the $\chi$-shift of the contragredient representation $V^*$ and
we denote by $F^{\vee}$ the $\chi$-shifted dual of $F$, where $\chi$ is the character introduced in Definition \ref{rho}.

Let $L=\prod_{j\in \Z} \g_j$ be a $\Z$-graded linearly compact Lie superalgebra with a $D$-conformal structure over $R$. 
Let $F$ be a finite-dimensional ${\mathfrak g}_{0}$-module  which we extend to
${L}_{\geq 0}=\prod_{j\geq 0}{\mathfrak g}_{j}$ by letting
${\mathfrak g}_{j}$, $j>0$, act trivially. We let
$$M(F)=\textrm{Ind}_{{ L}_{\geq 0}}^{{L}}F$$
be the parabolic finite Verma module, attached to $F$. 
It is natural to wonder whether the dual of a parabolic Verma module is still a parabolic Verma module. 
In \cite{CCK} we gave an answer to this question in the following theorem.
\begin{theorem}\label{main} Let $L$ be a Lie superalgebra with a $D$-conformal structure over $R$ and let $F$ be a finite-dimensional $\mathfrak{g}_0$-module. Then $M(F)^\vee$ is isomorphic to $M(F^\vee)$ as an $L$-module.
\end{theorem}

Theorem \ref{main} leads to the following natural problem for a $\Z$-graded linearly compact Lie superalgebra $L$: does $L$ have a $D$-conformal structure over a Lie conformal superalgebra $R$ of type $(r,s)$?
In \cite{CCK} we provided a positive answer to this question for the Lie superalgebras 
$W(r,s)$,  $E(3,6)$, $E(5,10)$ and $E(3,8)$ with the principal $\Z$-grading, while the case of $E(1,6)$ was essentially established in \cite{BKL1}. The only missing exceptional Lie superalgebra was therefore $E(4,4)$ and the main result of this paper is to provide a positive answer also in this case. 

\section{The Lie conformal superalgebras $RW(4,4)$ and $RE(4,4)$}
The prototypical example of a $\Z$-graded Lie conformal superalgebra of type $(r,s)$ is $RW(r,s)$ which is defined as follows.
 
 \begin{definition} 
  We denote by $RW(r,s)$ the free $\inlinewedge[\deb]$-module with even generators $a_1,\ldots,a_{r}$ and odd generators $a_{r+1},\ldots,a_{r+s}$, all of degree $-2$, and the $\la$-bracket given by\[
	[{a_i}_\la a_j]=(\de_i+\lambda_i)a_j+a_i\lambda_j,\hspace{0.5cm} i,j=1,\ldots,r+s\]
	and extended on the whole $RW(r,s)$ by the sesquilinearity \eqref{defsupconf1} of Definition \ref{defsupconf}.
\end{definition}

In \cite{CCK} we proved the following results:
\begin{proposition}\label{Wiso}\cite[Propositions 4.2, 4.3]{CCK}
The Lie superalgebra $W(r,s)$  consisting of all continuous derivations of $\C[[\xb]]$
 has  a $D$-conformal structure over $RW(r,s)$.
The map
	\[
	 \varphi: W(r,s)\rightarrow \mathcal A(RW(r,s)),
	\]
	given by $  \xb_K \de_{x_i}\mapsto -\yb_K a_i$
	is a $\Z$-graded Lie superalgebra isomorphism, where the $\Z$-gradation on the Lie superalgebra $W(r,s)$ of all continuous derivations of $\C[[\xb]]$ is given by $\deg x_i=-\deg \de_{x_i}=2$.
\end{proposition}

	Now we introduce a notable Lie conformal subalgebra $RE(4,4)$ of $RW(4,4)$. To this aim we introduce the following notation: if $u=(u_1,\dots,u_8)$ and 
 $v=(v_1,\dots,v_8)$ 
 we set
 \begin{equation}\label{eleven}B_{uv}=\sum_{i=1}^4u_iv_{i+4},\,\, C_{uv}=\sum_{i=1}^4(u_iv_{i+4}+u_{i+4}v_i).\end{equation}
 For example:
\[B_{\de\de}= \sum_{i=1}^4 \de_{i}\de_{i+4}\in \C[\deb],\hspace{5mm}
	C_{\de a}=\sum_{i=1}^4 (\de_ia_{i+4}+\de_{i+4}a_i)\in RW(r,s).\hspace{5mm}\]

 This is the main definition:
	\begin{definition} For $i=1,\ldots,8$ we let
		
		\begin{equation}\label{twelve}
		\alpha_i=\de_i B_{\de\de} C_{\de a} -2 \nu_i C_{\de a} + \de_i B_{\nu a}\in RW(4,4),\end{equation}

\noindent
where
\[\nu_1= \de_{6}\de_7\de_8,\,\,\,\,
\nu_2=- \de_{5}\de_7\de_8,\,\,\,\,
\nu_3= \de_{5}\de_6\de_8,\,\,\,\,
\nu_4=-\de_{5}\de_6\de_7,\,\,\,\,
\nu_5=\nu_6=\nu_7=\nu_8=0.\]
Denote by $RE(4,4)$ the $\C[\de]$-submodule of $RW(4,4)$ generated by the elements $\alpha_1,\ldots,\alpha_8$.
	
\end{definition}
Note that all elements $\alpha_1,\ldots,\alpha_8$ have degree $-10$, $\alpha_1,\ldots,\alpha_4$ are even and $\alpha_5,\ldots,\alpha_8$ are odd. As an example, for the reader's convenience, we explicitly write  the elements $\alpha_1$ and $\alpha_5$:
\begin{align*}
	\alpha_1
	&= 2\de_{5}\de_{6}\de_{7}\de_{8}a_{1}-\de_{1}\de_{6}\de_{7}\de_{8}a_{5}-\de_{1}\de_{5}\de_{7}\de_{8}a_{6}+\de_{1}\de_{5}\de_{6} \de_{8}a_{7}-\de_{1}\de_{5}\de_{6}\de_{7} a_{8} - 2\de_{2}\de_{6}\de_{7}\de_{8} a_{6} \\
 &-2\de_{3}\de_{6}\de_{7}\de_{8}a_{7} -2\de_{4}\de_{6}\de_{7}\de_{8} a_{8} + \de_{1}^2\de_{5}\de_{8}a_{4} + \de_{1}^2\de_{5}\de_{6}a_{2} + \de_{1}^2\de_{5}\de_{7} a_{3} +\de_{1}\de_{2} \de_{6}\de_{7}	a_{3} \\
 &-\de_{1}\de_{2}\de_{5}\de_{6}a_{1} + \de_{1}\de_{2}\de_{6}\de_{8} a_{4}
	+\de_{1}\de_{3} \de_{7}\de_{8} a_{4} - \de_{1}\de_{3}\de_{5}\de_{7}a_{1} -\de_{1}\de_{3}\de_{6}\de_{7} a_{2} - \de_{1}\de_{4} \de_{7}\de_{8}a_{3}\\
 &-\de_{1}\de_{4}\de_{6}\de_{8}a_{2} - \de_{1}\de_{4}\de_{5}\de_{8} a_{1} + \de_{1}^3\de_{5} a_{5} + \de_{1}^2 \de_{2} \de_{5} a_{6} + \de_{1}^2\de_{3} \de_{5} a_{7} + \de_{1}^2\de_{4} \de_{5} a_{8} \\
 &+\de_{1}^2\de_{2} \de_{6} a_{5} + \de_{1}\de_{2}^2 \de_{6} a_{6} +\de_{1}\de_{2}\de_{3} \de_{6} a_{7} + \de_{1}\de_{2}\de_{4} \de_{6} a_{8} + \de_{1}^2\de_{3} \de_{7} a_{5} + \de_{1}\de_{2}\de_{3} \de_{7} a_{6}\\
 &+\de_{1}\de_{3}^2 \de_{7} a_{7}+\de_{1}\de_{3}\de_{4} \de_{7} a_{8} + \de_{1}^2\de_{4} \de_{8} a_{5}+\de_{1}\de_{2}\de_{4} \de_{8} a_{6}+\de_{1}\de_{3}\de_{4} \de_{8} a_{7}+\de_{1}\de_{4}^2 \de_{8} a_{8};
\end{align*}
\begin{align*}
\alpha_5=&\de_5\de_6\de_7\de_8 a_5+ \de_2\de_5\de_6\de_7 a_3+\de_1 \de_2\de_5\de_6 a_5+\de_2 \de_5 \de_6 \de_8 a_4+\de_3 \de_5 \de_7 \de_8 a_4+\de_1 \de_3 \de_5 \de_7 a_5\\
 &-\de_3 \de_5 \de_6 \de_7 a_2-\de_4 \de_5 \de_7 \de_8 a_3-\de_4 \de_5 \de_6 \de_8 a_2+\de_1 \de_4 \de_5 \de_8 a_5+\de_2^2 \de_5 \de_6  a_6+\de_2 \de_3 \de_5 \de_6 a_7\\
 &+\de_2 \de_4 \de_5 \de_6 a_8+\de_ 2\de_3 \de_5 \de_7 a_6+\de_3^2 \de_5 \de_7  a_7+\de_3 \de_4 \de_5 \de_7 a_8+\de_2 \de_4 \de_5 \de_8 a_6+\de_3 \de_4 \de_5 \de_8 a_7\\
 &+\de_4^2 \de_5 \de_8 \de a_8.
\end{align*}
\begin{proposition}\label{relations}
For all $i,j=5,6,7,8$, with $i\neq j$, we have the following relations in $RE(4,4)$:
\begin{enumerate}
	\item $\de_{i}\alpha_{i}=0$;
	\item $\de_{i} \alpha_{j}=-\de_{j} \alpha_{i}$;
	\item $\de_{i}\alpha_{j-4}=\de_{j-4} \alpha_{i}$;
	\item $\de_{i} \alpha_{i-4}=\de_{i-4} \alpha_{i}-2 \sum_{k=1}^4 \de_k \alpha_{k+4}$;
 \item $2\de_i \alpha_j=\de_{p}\alpha_{q}-\de_{q}\alpha_{p}$, where $p,q\in \{1,2,3,4\}$ are such  that $\epsilon_{i-4\,j-4\,p\,q}=1$.
\end{enumerate}
\end{proposition}
\begin{proof}
Recalling that for $i\geq 5$ we have $\alpha_i=\de_i(B_{\de \de}C_{\de a}+B_{\nu a})$, Equations (1) and (2)  trivially hold since $\de_i^2=0$ and $\de_i \de_j=-\de_j \de_i$.
Observing that $\de_i\nu_{j-4}=0$, we obtain (3):
\[\de_i \alpha_{j-4}=\de_i \de_{j-4}(B_{\de \de}C_{\de a}+B_{\nu a})-2\de_i \nu_{j-4} C_{\de a}=\de_{j-4} \alpha_i.
\]
Furthermore, we observe that
\[
\de_5 \nu_1=\de_6 \nu_2=\de_7 \nu_3=\de_8 \nu_4=\de_5 \de_6 \de_7 \de_8  
\]
and we compute
\[
\de_i \alpha_{i-4}=\de_i\de_{i-4}(B_{\de\de}C_{\de a}+B_{\nu a})-2\de_i \nu_{i-4}C_{\de a}=\de_i \alpha_{i-4}-2 \de_5\de_6 \de_7 \de_8 C_{\de a}.
\]
On the other hand
\[
\sum_{k=1}^4 \de_k \alpha_{k+4}=\sum_k \de_k \de_{k+4}(B_{\de \de}C_{\de a}+B_{\nu a})=B_{\de \de}(B_{\de \de}C_{\de a}+B_{\nu a})=B_{\de \de} B_{\nu a}=\de_5\de_6 \de_7 \de_8 C_{\de a},
\]
hence Equation (4) follows.
Finally we compute
\[
\de_{i}\alpha_j=\de_i\de_j(B_{\de\de}C_{\de a}+B_{\nu a})=\de_i \de_j B_{\de \de} C_{\de a}=\de_i \de_j (\de_p\de_{p+4}+\de_q\de_{q+4})C_{\de a}
\]
and, observing that $\nu_q=-\de_i\de_j\de_{p+4}$ and $\nu_p=\de_i\de_j\de_{q+4}$,
\begin{align*}
\de_{p}\alpha_{q}-\de_{q}\alpha_{p}&=\de_{p}\de_{q}B_{\de \de} C_{\de a} -2\de_{p}\nu_{q}C_{\de a} +\de_{p}\de_{q} B_{\nu a}\\
&\hspace{4mm}-\de_{q}\de_{p}B_{\de \de} C_{\de a} +2\de_{q}\nu_{p}C_{\de a} -\de_{q}\de_{p} B_{\nu a}\\
&=-2 (\de_{p}\nu_{q}-\de_{q}\nu_{p})C_{\de a}\\
&=2 \de_i \de_j (\de_p\de_{p+4}+\de_q\de_{q+4})C_{\de a},
\end{align*}
and Equation (5) follows.

\end{proof}
\begin{corollary}\label{free4}
We have that $RE(4,4)$ is a free $\C[\de_1,\ldots,\de_4]$-module with basis $\alpha_1,\ldots,\alpha_8$. In particular there are no relations in $RE(4,4)$ among $\alpha_1,\ldots,\alpha_8$ which are independent from those in Proposition \ref{relations}.
\end{corollary}
\begin{proof}
Relations (1), (3), (4), (5) in Proposition \ref{relations} (actually (2) is a consequence of (5)) allow us to write every element $\beta\in RE(4,4)$ in the
form
\[
\beta=\sum_{i=1}^8 f_i(\de_1,\de_2,\de_3,\de_4)\alpha_i
\]
where $f_i(\de_1,\de_2,\de_3,\de_4)\in \C[\de_1,\ldots,\de_4]$ for all $i=1,\ldots,8$.
We have to show that if $\beta=0$ then $f_i=0$ for all $i$. If we expand all $\alpha_i$'s in terms of the elements $a_i$'s and we keep all terms divisible by $\de_5\de_6\de_7\de_8$ we obtain
\[
\de_5\de_6\de_7\de_8(2f_1a_1+2f_2a_2+2f_3a_3+2f_4a_4+f_5a_5+f_6a_6+f_7a_7+f_8a_8)=0.
\]
Therefore $f_1=f_2=\cdots=f_8=0$ since $a_1,\ldots,a_8$ are free generators of $RW(4,4)$.
\end{proof}
Before stating the main result we need some further notation. We let
\[\mu_1=\lambda_{6}\lambda_7\lambda_8,\,\,\,\mu_2=-\lambda_{5}\lambda_7\lambda_8,\,\,\, \mu_3=\lambda_{5}\lambda_6\lambda_8,\,\,\, \mu_4=-\lambda_{5}\lambda_6\lambda_7,\,\,\mu_5=\mu_6=\mu_7=\mu_8=0\]
and, consistently with \eqref{eleven}, we will consider the following elements:
\[	B_{\lambda \lambda}= \sum_{i=1}^4 \lambda_i \lambda_{i+4}, \hspace{1mm}
	C_{\lambda\de}=\sum_{i=1}^8 \lambda_i \de_{i+4},\hspace{1mm}
	B_{\mu\de}=  \sum_{i=1}^4 \mu_i \de_{i+4},\hspace{1mm}
	C_{\lambda\alpha}= \sum_{i=1}^8 \lambda_i \alpha_{i+4},\hspace{1mm}
	B_{\mu\alpha}= \sum_{i=1}^4 \mu_i \alpha_{i+1}.
\]
	\begin{theorem}
		For all $i,j=1,\ldots,8$ we have
		\begin{align*}\label{labre44}{[\alpha_i}_{\lab} \alpha_j]& = (\lambda_{i}\lambda_j B_{\lambda\lambda}  -2 \lambda_i \mu_j  -2 \lambda_j \mu_i) C_{\lambda\alpha}+3 \lambda_{i}\lambda_j B_{\mu\alpha}
			 + (\lambda_i C_{\mu\de} + \lambda_i B_{\lambda\lambda} C_{\lambda\de} 
    -2 \mu_i C_{\lambda\de}) \alpha_j.
		\end{align*}
		In particular  the $\inlinewedge[\deb]$-submodule $RE(4,4)$  of $RW(4,4)$ generated by $\alpha_1,\ldots,\alpha_8$ is a Lie conformal subalgebra of $RW(4,4)$.
	\end{theorem}
\begin{proof} The proof of this theorem consists of a finite number of computations that we verified with a computer. Recalling the definition of $\alpha_i$ in \eqref{twelve}, the following formulas can be useful for a direct check:
\begin{equation} \label{cnucnu}
[{B_{\nu a}}_\lab B_{\nu a}]=-B_{\mu\de} B_{\nu a},
\end{equation}

\begin{equation}
\label{cdecde}
[{C_{\de a}}_\lab C_{\de a}]=-(2B_{\lambda \lambda}+C_{\lambda \de}) C_{\de a},
\end{equation}

\begin{equation}
\label{cnucde}
[{B_{\nu a}}_\lab C_{\de a}]=\frac{1}{2}C_{\lambda\mu}C_{\lambda a}-B_{\mu \de} C_{\de a}-C_{\lambda \de} B_{\mu a} - B_{\mu \de} C_{\lambda a} +C_{\lambda \mu} C_{\de a},
\end{equation}

\begin{align}
\label{cdecnu}
[{C_{\de a}}_\lab B_{\nu a}]=&-C_{\lambda \de } B_{\rho a}-C_{\lambda \de}B_{\nu a}-2B_{\lambda \lambda} B_{\nu a}+C_{\lambda \nu} C_{\lambda a}+C_{\lambda \de} B_{\mu a}-B_{\mu \de} C_{\lambda a}\\
\nonumber&-\frac{1}{2}C_{\lambda \mu} C_{\de a}+\frac{1}{2}C_{\lambda \mu} C_{\lambda a}+2B_{\mu \de} C_{\de a}+C_{\lambda \de} B_{\sigma a}+2B_{\de \de } B_{\mu a},
\end{align}
where \begin{align*}
   &\rho_1=\lambda_6\de_7\de_8+\de_6\lambda_7\de_8+
\de_6\de_7\lambda_8, \,\,\, \rho_2=-(\lambda_5\de_7\de_8+\de_5\lambda_7\de_8+
\de_5\de_7\lambda_8),\\
&\rho_3=\lambda_5\de_6\de_8+\de_5\lambda_6\de_8+
\de_5\de_6\lambda_8,\,\, \, \rho_4=-(\lambda_5\de_6\de_7+\de_5\lambda_6\de_7+
\de_5\de_6\lambda_7),\\
&\rho_5=\rho_6=\rho_7=\rho_8=0,
\end{align*}
and 
 \begin{align*}
    \sigma_1&=\lambda_6\lambda_7\de_8+\de_6\lambda_7\lambda_8+
\lambda_6\de_7\lambda_8,\, \sigma_2=-(\lambda_5\lambda_7\de_8+\de_5\lambda_7\lambda_8+
\lambda_5\de_7\lambda_8),\\
\sigma_3&=\lambda_5\lambda_6\de_8+\de_5\lambda_6\lambda_8+
\lambda_5\de_6\lambda_8,\,
\sigma_4=-(\lambda_5\lambda_6\de_7+\de_5\lambda_6\lambda_7+
\lambda_5\de_6\lambda_7),\\
\sigma_5&=\sigma_6=\sigma_7=\sigma_8=0.
\end{align*}
As an example we
compute $[{\alpha_5}_\lab \alpha_5]$, recalling that $\alpha_5=\de_5 B_{\de \de} C_{\de a}+\de_5 B_{\nu a}$.
By Equation \eqref{cnucnu} we have
\begin{equation}\label{alph51}
[{\de_5 B_{\nu a}}_\lab \de_5 B_{\nu a}]=\lambda_5 \de_5 B_{\mu \de}B_{\nu a}=0
\end{equation}
since every term in $B_{\mu \de}$ either contains $\lambda_5 $ or $\de_5$.
By Equation \eqref{cdecde}, the properties of $\lab$-products, and recalling that $B_{\lambda \lambda}$, $C_{\lambda \de}$ and $B_{\de \de}$ are odd elements, we have
\begin{align}
\label{alph52}
[{\de_5 B_{\de \de} C_{\de a}}_\lab \de_5 B_{\de \de }C_{\de a}]&=\lambda_5 \de_5 C_{\lambda \lambda} (C_{\lambda \lambda}+C_{\lambda \de}+B_{\de \de})[{C_{\de a}}_\lab C_{\de a}]\\
\nonumber&=-\lambda_5 \de_5 C_{\lambda \lambda} B_{\de \de} C_{\lambda \de}C_{\de a}
\end{align}
since $B_{\lambda \lambda}^2=0=C_{\lambda \de}^2$.
Similarly we can compute, by Equation \eqref{cnucde},
\begin{align}\label{alph53}
[{\de_5 C_{\nu a}}_\lab \de_5 B_{\de \de }C_{\de a}]&=-\lambda_5 \de_5 (B_{\lambda \lambda}+C_{\lambda \de}+B_{\de \de}) [{B_{\nu a}}_\lab C_{\de a}]\\
\nonumber &=\lambda_5 \de_5 B_{\de \de} C_{\lambda \de} B_{\mu a},
\end{align}
since $\lambda_5C_{\lambda \mu}=0$, $\lambda_5\de_5B_{\mu \de}=0$,
$\lambda_5B_{\lambda \lambda}B_{\mu a}=0$,
and by Equation \eqref{cdecnu}
\begin{align}\label{alph54}
[{\de_5 B_{\de \de} C_{\de a}}_\lab \de_5 B_{\nu a}]&=-\lambda_5 \de_5 B_{\lambda \lambda} [{C_{\de a}}_\lab B_{\nu a}]\\
\nonumber &=\lambda_5 \de_5 B_{\lambda \lambda} C_{\lambda \de} B_{\rho a}+\lambda_5 \de_5 B_{\lambda \lambda} C_{\lambda \de} B_{\nu a}
\end{align}
since  $\lambda_5\de_5C_{\lambda \nu}=0$, $\lambda_5C_{\lambda \lambda}B_{\mu \de}=0$ and $\lambda_5\de_5 B_{\lambda \lambda} C_{\lambda \de}B_{\sigma a}=0$.
We have, by Equations \eqref{alph51},\eqref{alph52},\eqref{alph53},\eqref{alph54},
\begin{align*}
[{\alpha_5}_\lab \alpha_5]=-\lambda_5\de_5B_{\lambda \lambda} B_{\de \de}C_{\lambda \de}C_{\de a} +\lambda_5 \de_5 B_{\de \de} C_{\lambda \de} B_{\mu a}+\lambda_5 \de_5 B_{\lambda \lambda} C_{\lambda \de} B_{\rho a}+\lambda_5 \de_5 B_{\lambda \lambda} C_{\lambda \de} B_{\nu a}.
\end{align*}
Now we can observe that all terms in $B_{\de \de} C_{\lambda \de} B_{\mu a}+ B_{\lambda \lambda} C_{\lambda \de} B_{\rho a}$ which are not divisible by $\lambda_5$ or $\de_5$ cancel out and so we have
\[
[{\alpha_5}_\lab \alpha_5]=-\lambda_5\de_5B_{\lambda \lambda} B_{\de \de}C_{\lambda \de}C_{\de a}+\lambda_5 \de_5 B_{\lambda \lambda} C_{\lambda \de} B_{\nu a}=\lambda_5 B_{\lambda \lambda} C_{\lambda \de} \alpha_5
\]
and the proof is complete, since one can easily verify that $\lambda_5 B_{\mu \de}\alpha_5=0$.
\end{proof}
\begin{corollary}\label{properties}
The Lie  superalgebra $\mathcal A(RE(4, 4))$ has a $D$-conformal structure over $RE(4,4)$.
\end{corollary}
\begin{proof}
${} $
We check all the conditions of Definition \ref{assumption}.
\begin{enumerate}
\item This is obvious.
\item $RE(4,4)$ is $\mathbb Z$-graded: it is generated by the homogeneous elements $\alpha_1,\ldots,\alpha_8$ of degree $-10$;
\item the induced $\mathbb Z$-gradation on $\mathcal A(RE(4,4))$ has depth 2: although elements $\alpha_i$ have degree -10 all elements of the form $y_jy_hy_k\alpha_i$ vanish in the annihilation superalgebra; indeed every summand in the element $\alpha_i$ involves four derivatives with respect to the $y$-variables.
\item The homogeneous components of $\mathcal A(R)$ of degree -1 and -3 vanish, and in particular they are purely odd;
\item for all $i=1,\ldots,8$ we have that $y_5y_6y_7y_8 \alpha_i$ is a nonzero scalar multiple of $a_i$ in $\mathcal A(RE(4,4))_{-2}$ (and therefore it acts as a scalar multiple of $\de_{y_i}$ on $\mathcal A(RE(4,4))$) and for every set of indices $i_1,i_2,i_3,i_4$ we have that $y_{i_1}y_{i_2}y_{i_3}y_{i_4} \alpha_i$ is a scalar multiple of $a_j$ for some $j$. Therefore the map $ad:\mathcal A(R)_{-2}\rightarrow der(\mathcal A(R))$ is injective and its image is $\langle \de_{y_1},\ldots,\de_{y_8}\rangle$.
\end{enumerate}		
\end{proof}
\begin{theorem}\label{generators}
The annihilation superalgebra $\mathcal A(RE(4,4))$ is
spanned by the elements $f(y_1,y_2,y_3,y_4)y_5y_6y_7y_8\alpha_i$
with $f(y_1,y_2,y_3,y_4)\in\C[y_1,y_2,y_3,y_4]$ and $i=1,\dots,8$.
\end{theorem}
\begin{proof} The result follows by repeated applications of the relations in Proposition \ref{relations}. We show an explicit example to illustrate this fact:
\begin{align*}
y_1y_2y_5y_6y_7\alpha_6&=y_1y_2y_5y_6y_7\de_{y_8}y_8\alpha_6=y_1y_2y_5y_6y_7y_8\de_8\alpha_6\\
&=\frac{1}{2}y_1y_2y_5y_6y_7y_8(\de_1\alpha_3-\de_3\alpha_1)\\
&=
-\frac{1}{2}\de_{y_1}(y_1y_2y_5y_6y_7y_8)\alpha_3=-\frac{1}{2}y_2y_5y_6y_7y_8\alpha_3.
\end{align*}
\end{proof}

\section{Embedding $E(4,4)$ into $W(4,4)$}
In this section we want to describe an explicit embedding  of the exceptional Lie superalgebra $E(4,4)$ into $W(4,4)$. As a consequence of this fact we show that $E(4,4)\cong \mathcal A(RE(4,4))$ and in particular that $E(4,4)$ has a $D$-conformal structure. 

Recall that
the Lie superalgebra $E(4,4)$ is defined as follows (\cite[\S 5.3]{CK2}): its even part $\g_{\bar{0}}$ is the Lie algebra $W_4$ of vector fields in four even indeterminates  with coefficients in the field of formal power series, and its odd part $E(4,4)_{\bar{1}}$ is isomorphic,  as a $E(4,4)_{\bar{0}}$-module, to $\Omega^1(4)^{-1/2}$. This means that 
$E(4,4)_{\bar{1}}$ is the space of all differential one-forms in four even indeterminates with coefficients in the field of formal power series and, for $X\in E(4,4)_{\bar{0}}$,
$\omega\in E(4,4)_{\bar{1}}$,
$$[X,\omega]=L_X(\omega)-\frac{1}{2}\diver(X)\omega,$$
where $L_X(\omega)$ denotes the Lie derivative of the one-form $\omega$ along the vector field $X$ (see also \cite[Definition 2.5]{CK}). Besides, for $\omega_1, \omega_2\in E(4,4)_{\bar{1}}$:
$$[\omega_1, \omega_2]=d\omega_1\wedge \omega_2+\omega_1\wedge d\omega_2$$
where three-forms are identified with vector fields via contraction with the standard volume form.

\begin{remark}\rm
In \cite[\S 2.1.3]{K77}, in the classification of finite-dimensional Lie superalgebras, it appears the Lie superalgebra $p(4)$ (denoted in \cite{K77}  by $p(3)$). This is
a simple finite-dimensional Lie superalgebra with the following  consistent irreducible  $\Z$-grading:
$$p(4)=p(4)_{-1}\oplus p(4)_0\oplus p(4)_1$$
where $p(4)_0\cong \frak{sl}_4$ and $p(4)_{-1}\cong \Lambda^2((\C^4)^*)$, 
$p(4)_1\cong S^2\C^4$, as $\frak{sl}_4$-modules. The Lie superalgebra $p(4)$ has a unique, up to isomorphisms, 
non-trivial central extension that we will denote by $\hat{p}(4)$ (\cite[Example 3.6]{K}, \cite{S}):
observe that $\Lambda^2((\C^4)^*)$ is isomorphic to the standard $\frak{so}_6$-module, with scalar product $(\cdot,\cdot)$; define 
$\varphi:\Lambda^2(p(4))\rightarrow \C$ by setting, for $x\in p(4)_i, y\in p(4)_j$,
$$\varphi(x,y):=\left\{\begin{array}{cc}
(x,y) & {\mbox{if}}\, i=j=-1\\
0 & {\mbox{otherwise.}}
\end{array}\right.$$
Then $\hat{p}(4)$ is the central extension of $p(4)$ defined by this cocycle.
Therefore one has the following $\Z$-graded Lie superalgebra
$$\hat{p}(4)=\hat{p}(4)_{-2}\oplus\hat{p}(4)_{-1}\oplus \hat{p}(4)_0\oplus \hat{p}(4)_1$$
with $\hat{p}(4)_{-2}\cong\C$, central.
\end{remark}

\bigskip

The Lie superalgebra $L=E(4,4)$ has, up to conjugation, only one irreducible $\Z$-grading $L=\prod_{j\geq -1} \g_j$, called the principal grading, defined by setting $\deg(x_i)=1$ and $\deg d=-2$ (\cite[Corollay 9.8]{CK}).
It is a grading of depth 1, where $\g_0$ is isomorphic to $\hat{p}(4)$ and
$\g_{-1}\cong \C^{4|4}$. In this isomorphism we have that $\hat{p}(4)_{1}$ is spanned by the one-forms $x_idx_j+x_jdx_i$ (with $i,j=1,2,3,4$), $\hat{p}(4)_0$ is spanned by the vector fields $x_i\de_{x_j}$ and $x_i\de_{x_i}-x_j\de_{x_j}$ (with $i,j=1,2,3,4$, $i\neq j$), $\hat{p}(4)_{-1}$ is spanned by the one-forms $x_idx_j-x_jdx_i$ (with $i,j=1,2,3,4)$ and $\hat{p}(4)_{-2}$ is spanned by $x_1\de_{x_1}+x_2\de_{x_2}+x_3\de_{x_3}+x_4\de_{x_4}$.

Let us fix the Borel subalgebra $\langle x_i\partial_j, h_{ij}=x_i\partial_i-x_j\partial_j ~|~ i<j\rangle$ of $(\g_0)_0=\frak{sl}_4$ and  consider the usual set of simple roots of the corresponding root system, given by $\{\alpha_{12},\alpha_{23},\alpha_{34}\}$. We let $\Lambda$ be the weight lattice of $\frak{sl}_4$ and we express all weights of $\frak{sl}_4$ using their coordinates with respect to the fundamental weights $\omega_{12},\omega_{23},\omega_{34}$, i.e., for $\lambda\in \Lambda$ we write $\lambda=(\lambda_{12},\lambda_{23},\lambda_{34})$ for some $\lambda_{i\,i+1}\in \mathbb Z$ to mean $\lambda=\lambda_{12}\omega_{12}+\lambda_{23}\omega_{23}+\lambda_{34}\omega_{34}$.
If $\lambda=(a,b,c)\in \Lambda$ is a dominant weight we shall denote by $F(\lambda)=F(a,b,c)$ the irreducible $\mathfrak{sl}_4$-module of highest 
weight $\lambda$.

The component $\g_1$ of the principal $\Z$-grading of $L$ is an irreducible $\g_0$-module and, as an $\frak{sl}_4$-module, it decomposes as follows:
$\g_1=V_1\oplus V_2\oplus V_3\oplus V_4$ where  
$V_1\cong F(2,0,1)$, $V_2\cong F(1,0,0)$, $V_3\cong F(3,0,0)$ and $V_4\cong F(1,1,0)$.

\begin{remark}\rm
We point out that one can construct a graded embedding of the Lie superalgebra $\g_0=\hat{p}(4)$ into the exceptional Lie superalgebra $E(5,10)$. (For the description of $E(5,10)$ we refer to \cite{K} and \cite{CCK2}). 
Indeed the following map
\begin{align*}
\sum_{i=1}^4x_i\de_{x_i}&\mapsto \frac{1}{2}\de_{x_5},\\
x_idx_j-x_jdx_i&\mapsto dx_i\wedge dx_j,\\
x_i\de_{x_j}&\mapsto x_i\de_{x_j}, \hspace{41mm}\mbox{for}\, i\neq j,\\
x_i\de_{x_i}-x_{i+1}\de_{x_{i+1}}&\mapsto x_i\de_{x_i}-x_{i+1}\de_{x_{i+1}},\hspace{20mm}\mbox{for}\,\, i=1,\dots 3,\\
x_idx_j+x_jdx_i&\mapsto x_idx_j\wedge dx_5+x_jdx_i\wedge dx_5,\\
\end{align*}
defines an embedding of  $\g_0=\hat{p}(4)$ into $E(5,10)$. However, this embedding does not extend to the whole $E(4,4)$ \cite{CK2}.
\end{remark}

\begin{definition} A $\Z$-graded Lie superalgebra 
$L=\prod \g_i$ is called transitive if it satisfies the following property: if $x\in \g_i$, $i\geq 0$,  is such that $[x,\g_{-1}]=0$, then $x=0$.
\end{definition}
We now recall the following simple embedding theorem which generalizes a standard result for Lie algebras (\cite{Blat}, \cite[\S 5.4]{K77}). 
\begin{theorem}
If ${L}=\prod_{i\geq -1}\g_i$ is a transitive $\Z$-graded Lie superalgebra with $\dim(\mathfrak{g}_{-1})_{\bar{0}}=m$, $\dim(\mathfrak{g}_{-1})_{\bar{1}}=n$, then there is an embedding 
\[\mathfrak{g}\rightarrow W(r,s)\]
mapping the $\Z$-grading of $L$ to the principal grading of $W(r,s)$.
\end{theorem}
In order to describe the embedding of $E(4,4)$ into $W(4,4)$ explicitly, we will denote by
$X=\{x_1,\dots,x_4\}$ the (even) variables of $E(4,4)$ and by $Y=\{y_1,\dots,y_4\}$ the corresponding (even) variables of $W(4,4)$; accordingly, we will denote by $\{y_5, \dots, y_8\}$ the odd variables of $W(4,4)$.
\begin{theorem}\label{Embedding}
The following map $\Phi$ defines an embedding of $E(4,4)$ into $W(4,4)$, preserving the principal grading:
for $r,i=1,\dots,4$,
	\begin{align*}
	f(X)\de_{x_r}&\mapsto f(Y)\de_{y_r}+\sum_{i=1}^4 \de_{y_i}f(Y)\beta_1(r,i)+\sum _{i=1}^4
\de_{y_r}\de_{y_i} f(Y) \beta_2(i)+\sum_{i,l=1}^{4} \de_{y_r}\de_{y_i}\de_{y_l} f(Y) \beta_3(i,l)\\
f(X) dx_i& \mapsto f(Y)\de_{y_{i+4}}+\sum_{j=1}^4 \de_{y_j}f(Y) \gamma_1(i,j)+\sum_{j,l=1}^4\de_{y_j}\de_{y_l}f(Y) \gamma_2(i,j,l),	\end{align*}

\medskip

\noindent
where

$
\beta_1(r,i)=\begin{cases}
	-y_{r+4}\de_{y_{i+4}}&\textrm{if } r\neq i,\\
	-y_{r+4}\de_{y_{r+4}}+\frac{1}{2} \sum_{l=5}^8y_{l}\de_{y_{l}}& \textrm{if }r=i,
\end{cases}$

\noindent
and

$\begin{array}{lcl}
\beta_2(i)&=&-\frac{1}{2}(y_{j+4}y_{h+4} \de_{y_k}+y_{h+4}y_{k+4}\de_{y_j}+y_{k+4}y_{j+4}\de_{y_h}) , 
\\
\\
\beta_3(i,l)&=&\frac{1}{2}y_{j+4}y_{h+4}y_{k+4}\de_{y_{l+4}} ,
\\
\\
\gamma_1(i,j)&=&y_{h+4}\de_{y_k}-y_{k+4}\de_{y_h} ,
\\
\\
\gamma_2(i,j,l)&=&-y_{h+4}y_{k+4}\de_{y_{l+4}}, 
\end{array}$

\medskip

\noindent
 for $i,j,h,k\in\{1,2,3,4\}$ such that $\epsilon(ijhk)=1$,
 and $l\in \{1,2,3,4\}$.
\end{theorem}
\begin{proof} We want to show that,  for every $a\in \g_i,
b\in\g_j$, for every $i,j\geq -1$,  we have:
\begin{equation}
\Phi([a,b])=[\Phi(a),\Phi(b)].
\label{isom}
\end{equation}
By the transitivity of the principal grading of the Lie superalgebra $W(4,4)$,
and using induction on $i+j$,
it is enough to prove (\ref{isom}) for $i=-1$.

For $r,s=1,\dots, 4$ we have:

\[[\de_{x_r}, f(X)\de_{x_s}]=(\de_{x_r}f)\de_{x_s},\,\,\,\,\,
[\de_{x_r}, fdx_s]=(\de_{x_r}f)dx_s\]

\[\Phi(\de_{x_r}(f)\de_{x_s})=(\de_{y_r}f)\de_{y_s}+\sum_{i=1}^4
(\de_{y_i}\de_{y_r}f)\beta_1(s,i)+\sum_{i=1}^4
(\de_{y_s}\de_{y_i}\de_{y_r}f)\beta_2(i)+\sum_{i,l=1}^4(\de_{y_i}\de_{y_s}\de_{y_l}
\de_{y_r}f)\beta_3(i,l).\]
It follows that
\[[\Phi(\de_{x_r}),\Phi( f(X)\de_{x_s})]=\Phi((\de_{x_r}f)\de_{x_s})\]
since, for $r,s,i,l=1,\dots, 4$,
\[
[\de_{y_r}, \beta_1(s,i)]=0,\,\,[\de_{y_r}, \beta_2(i)]=0,\,\,[\de_{y_r}, \beta_3(i,l)]=0.
\]
Similarly, for $r=1,\dots,4$,

\[[\Phi(\de_{x_r}), \Phi(fdx_s)]=\Phi((\de_{x_r}f)dx_s)\]
since, for $r,s,j,l=1,\dots,4$, we have:
\[[\de_{y_r}, \gamma_1(s,j)]=0,\,\, [\de_{y_r}, \gamma_2(s,j,l)]=0.
\]
Besides, for $r\neq s$,
\[[dx_s, f(X)\de_{x_r}]=\frac{1}{2}
(\de_{x_r}f)dx_s,\]
 and
\[[\de_{y_{s+4}}, \Phi(f(X)\frac{\de}{\de{x_r}})]=\frac{1}{2}(\de_{y_r}f\de_{y_{s+4}}
+\frac{1}{2}\sum_{j=1}^4\de_{y_j}\de_{y_r}f\gamma_1(s,j)+
\sum_{l,j=1}^4\de_{y_j}\de_{y_l}\de_{y_r}f\gamma_2(s,j,l)),\]
since
\begin{align*}
[\de_{y_{s+4}}, \beta_1(r,i)]&=\frac{1}{2}\delta_{i,r}\de_{y_{s+4}}\\
[\de_{y_{s+4}}, \beta_2(i)]&=\gamma_1(s,i)\\
[\de_{y_{s+4}}, \beta_3(i,l)]&=\gamma_2(s,i,l).
\end{align*}
Similarly, for $r=1,\dots,4$,
\[\Phi([dx_r, f(X)\de_{x_r}])=[\Phi(dx_r), \Phi(f(X)\de_{x_r}].\]
Finally,
\[\Phi([dx_r, f(X)dx_r])=0=[\de_{y_{r+4}},\Phi(f(X)dx_r),\]
since
$\de_{y_{r+4}}\gamma_1(r,j)=0=\de_{y_{r+4}}\gamma_2(r,j,l)$,
and, for $i\neq j$, 
\[[dx_j,f(X)dx_i]=-(\de_{x_h}f)\de_{x_k}+(\de_{x_k}f)\de_{x_h}\]
\begin{align*}
[\de_{y_{j+4}}, \Phi(f(X)dx_i)]&=-(\de_{y_h}f)\de_{y_k}+(\de_{y_k}f)\de_{y_h}+
\sum_{l=1}^4(\de_{y_h}\de_{y_l}f)y_{k+4}\de_{y_{l+4}}-\sum_{l=1}^4(\de_{y_k}\de_{y_l}f)y_{h+4}\de_{y_{l+4}}\\
&=-(\de_{y_h}f)\de_{y_k}+(\de_{y_k}f)\de_{y_h}+
\sum_{l=1}^4(\de_{y_h}\de_{y_l}f)\beta_1(k,l)-\sum_{l=1}^4(\de_{y_k}\de_{y_l}f)\beta_1(h,l)\\
&=\Phi(-(\de_{x_h}f)\de_{x_k}+(\de_{x_k}f)\de_{x_h}).
\end{align*}
\end{proof}
\begin{example} For the convenience of the reader we explicitly write the image of the vector field $f(x_1,x_2,x_3,x_4)\de_{x_1}$ and of the form $f(x_1,x_2,x_3,x_4)dx_1$ under the map $\Phi$:

\begin{align*}	
	f(x_1,x_2,x_3,x_4)\de_{x_1}\mapsto &f(Y)\de_{y_1}+\frac{1}{2}\de_{y_1}f(Y)(-y_5\de_{y_5}+y_6\de_{y_6}+y_7\de_{y_7}+y_8 \de_{y_8})\\
	&-\de_{y_2}f(Y)y_5 \de_{y_6}-\de_{y_3}f(Y)y_5 \de_{y_7}-\de_{y_4}f(Y)y_5 \de_{y_8}\\
	&+\frac{1}{2}\de_{y_1}\de_{y_1} f(Y)(-y_6y_7\de_{y_4}-y_7y_8\de_{y_2}+y_6y_8 \de_{y_3})\\
	&+\frac{1}{2}\de_{y_1}\de_{y_2} f(Y)(-y_5y_8
	\de_{y_3}+y_7y_8\de_{y_1}+y_5y_7 \de_{y_4})\\
	&+\frac{1}{2}\de_{y_1}\de_{y_3} f(Y)(-y_5y_6
	\de_{y_4}-y_6y_8\de_{y_1}+y_5y_8 \de_{y_2})\\
	&+\frac{1}{2}\de_{y_1}\de_{y_4} f(Y)(y_5y_6
	\de_{y_3}-y_5y_7\de_{y_2}+y_6y_7 \de_{y_1})\\
	&+\frac{1}{2}\de_{y_1}\de_{y_1}\de_{y_1}f(Y) y_6y_7y_8 \de_{y_5}
	+\frac{1}{2}\de_{y_1}\de_{y_1}\de_{y_2}f(Y) y_6y_7y_8 \de_{y_6}\\
	&+\frac{1}{2}\de_{y_1}\de_{y_1}\de_{y_3}f(Y) y_6y_7y_8 \de_{y_7}
	+\frac{1}{2}\de_{y_1}\de_{y_1}\de_{y_4}f(Y) y_6y_7y_8 \de_{y_8}\\
	&-\frac{1}{2}\de_{y_1}\de_{y_2}\de_{y_1}f(Y) y_5y_7y_8 \de_{y_5}
	-\frac{1}{2}\de_{y_1}\de_{y_2}\de_{y_2}f(Y) y_5y_7y_8 \de_{y_6}\\
	&-\frac{1}{2}\de_{y_1}\de_{y_2}\de_{y_3}f(Y) y_5y_7y_8 \de_{y_7}
	-\frac{1}{2}\de_{y_1}\de_{y_2}\de_{y_4}f(Y) y_5y_7y_8 \de_{y_8}\\
	&+\frac{1}{2}\de_{y_1}\de_{y_3}\de_{y_1}f(Y) y_5y_6y_8 \de_{y_5}
	+\frac{1}{2}\de_{y_1}\de_{y_3}\de_{y_2}f(Y) y_5y_6y_8 \de_{y_6}\\
	&+\frac{1}{2}\de_{y_1}\de_{y_3}\de_{y_3}f(Y) y_5y_6y_8 \de_{y_7}
	+\frac{1}{2}\de_{y_1}\de_{y_3}\de_{y_4}f(Y) y_5y_6y_8 \de_{y_8}\\
	&-\frac{1}{2}\de_{y_1}\de_{y_4}\de_{y_1}f(Y) y_5y_6y_7 \de_{y_5}
	-\frac{1}{2}\de_{y_1}\de_{y_4}\de_{y_2}f(Y) y_5y_6y_7 \de_{y_6}\\
	&-\frac{1}{2}\de_{y_1}\de_{y_4}\de_{y_3}f(Y) y_5y_6y_7 \de_{y_7}
	-\frac{1}{2}\de_{y_1}\de_{y_4}\de_{y_4}f(Y) y_5y_6y_7 \de_{y_8}.
\end{align*}
and
\begin{align*}
	f(X)dx_1\mapsto
	&=f(Y)\de_{y_5}
	+\de_{y_2}f(Y)(y_7\de_{y_4}-y_8\de_{y_3})\\
	&+\de_{y_3}f(Y)(y_8\de_{y_2}-y_6\de_{y_4})+\de_{y_4}f(Y)(y_6\de_{y_3}-y_7\de_{y_2})\\
	&+\de_{y_2}\de_{y_1}f(Y)(-y_7y_8\de_{y_5})+\de_{y_2}\de_{y_2}f(Y)(-y_7y_8\de_{y_6})\\
	&+\de_{y_2}\de_{y_3}f(Y)(-y_7y_8\de_{y_7})+\de_{y_2}\de_{y_4}f(Y)(-y_7y_8\de_{y_8})\\
	&+\de_{y_3}\de_{y_1}f(Y)(y_6y_8\de_{y_5})+\de_{y_3}\de_{y_2}f(Y)(y_6y_8\de_{y_6})\\
	&+\de_{y_3}\de_{y_3}f(Y)(y_6y_8\de_{y_7})+\de_{y_3}\de_{y_4}f(Y)(y_6y_8\de_{y_8})\\
	&+\de_{y_4}\de_{y_1}f(Y)(-y_6y_7\de_{y_5})+\de_{y_4}\de_{y_2}f(Y)(-y_6y_7\de_{y_6})\\
	&+\de_{y_4}\de_{y_3}f(Y)(-y_6y_7\de_{y_7})+\de_{y_4}\de_{y_4}f(Y)(-y_6y_7\de_{y_8}).\\
	\end{align*}
\end{example}

\begin{corollary}\label{iso} The map $\Psi$:
\[f(x_1,x_2,x_3,x_4)\de_{x_i}\mapsto -\frac{1}{2} f(y_1,y_2,y_3,y_4)y_5y_6y_7y_8 \alpha_i \mod \tilde \de \tilde{R}E(4,4),\]
\[f(x_1,x_2,x_3,x_4)dx_i\mapsto -f(y_1,y_2,y_3,y_4)y_5y_6y_7y_8 \alpha_{i+4}\mod \tilde \de \tilde{R}E(4,4),\]
defines an isomorphism between $E(4,4)$ and 
$\mathcal A(RE(4,4))$. In particular Verma modules over
the Lie superalgebra $E(4,4)$ satisfy the duality property.
The shift $\chi$ defined in Definition \ref{rho} is 0.
\end{corollary}
\begin{proof} One can check that the map $\Psi=\varphi\circ \Phi$ is the composition of the maps $\varphi$ and $\Phi$ defined in Proposition \ref{Wiso} and Theorem \ref{Embedding}, respectively, where $\mathcal A(RE(4,4))$ is canonically identified with a Lie subalgebra of $\mathcal A(RW(4,4))$. The image of $\Psi$ is indeed ${\mathcal A}(RE(4,4))$ by Theorem \ref{generators}.
The shift is 0, since $E(4,4)_0$  coincides with its derived subalgebra.
\end{proof}

%


\medskip


\begin{thebibliography}{9}
 \bibitem{Blat} {\sc R. J. Blattner,}
 {\em Induced and produced representations of Lie algebras,}
 {Trans. Amer. Math. Soc. 144 (1969), 457--474}
 
\bibitem{BKL1} {\sc C.\ Boyallian, V.\ G.\ Kac, J.\ I.\ Liberati,}
	{\em Classification of finite irreducible modules over the Lie conformal superalgebra $CK_6$,}
	{Comm.\ Math.\ Phys.\ {\bf 317} (2013), 503--546.}

	\bibitem{BKLR} {\sc C.\ Boyallian, V.\ G.\ Kac, J.\ I.\ Liberati, A.\ Rudakov}
	{\em Representations of simple finite Lie conformal superalgebras of type $W$ and $S$,}
	{J.\ Math.\ Phys.\ {\bf 47} (2006), 1--25.}

\bibitem{CC}
{\sc
N.\ Cantarini, F.\ Caselli}
\newblock{\em Low Degree Morphisms of $E(5,10)$-generalized Verma
Modules,}
\newblock {Alg.\  Rep.\ Theory
 \newblock {\bf 23} (2020), 2131--2165.}

\bibitem{CCK}
{\sc 
N.\ Cantarini, F,\ Caselli, V.\ G.\ Kac,}
\newblock{\em  Lie conformal superalgebras and duality of modules over linearly compact Lie superalgebras}
\newblock Adv.\ Math.\ {\bf 378} (2021), 107523, 45pp.

\bibitem{CCK2}
{\sc 
N.\ Cantarini, F,\ Caselli, V.\ G.\ Kac,}
\newblock{\em   Classification of degenerate Verma modules for $E(5,10)$}
\newblock Comm. Math. Phys. {\bf 385} (2021), 963--1005.

\bibitem{CK}
{\sc 
N.\ Cantarini, V.\ G.\ Kac,}
\newblock{\em  Infinite dimensional primitive linearly compact Lie superalgebras,}
\newblock Adv.\ Math.\ {\bf 207(1)} (2006), 328-419.


\bibitem{CK0}
	{\sc
		S.-J.\ Cheng, V.\ G.\ Kac,}
	\newblock{\em Conformal modules,}
	\newblock Asian J.\ Math.\
	{\bf 1} (1997), 181--193.

\bibitem{CK2}
{\sc
S.-J.\ Cheng, V.\ G.\ Kac}
\newblock{\em Structure of some $\Z$-graded Lie superalgebras of vector fields,}
\newblock Transf.\ Groups\
{\bf 4} (1999), 219--272.

\bibitem{K77}
{\sc V.\ G.\ Kac,}
\newblock{\em Lie superalgebras,}
\newblock Adv.\ Math.\ {\bf 26} (1977), 8--96.

\bibitem{K}
{\sc V.\ G.\ Kac,}
\newblock{\em Classification of infinite-dimensional simple linearly 
compact Lie superalgebras,}
\newblock Adv.\ Math.\ {\bf 139} (1998), 1--55.

\bibitem{K1} {\sc V.\ G.\ Kac, } {\em Vertex Algebras for Beginners,} University Lecture Series 10, 2nd Ed., AMS, Providence, RI, 1998.

\bibitem{S}
{\sc V.\ Serganova,}
\newblock{\em Representations of a central extension of the simple Lie superalgebra $p(3)$,}
\newblock Sao Paulo J.\ of Math.\ Sciences {\bf 12(2)} (2018), 1--15.







\end{thebibliography}
\end{document}